\documentclass[a4paper,11pt]{article}
\usepackage[english]{babel}
\usepackage[centertags]{amsmath}
\usepackage{amsfonts}
\usepackage{amssymb}
\usepackage{amsthm}

\newtheorem{theorem}{Theorem}[section]
\newtheorem{lemma}[theorem]{Lemma}
\newtheorem{proposition}[theorem]{Proposition}
\newtheorem{cor}[theorem]{Corollary}

\theoremstyle{definition}
\newtheorem{definition}[theorem]{Definition}
\newtheorem{example}[theorem]{Example}

\theoremstyle{remark}
\newtheorem{remark}[theorem]{Remark}

\numberwithin{equation}{section}

\begin{document}
\begin{center}
\Large{\textbf{Graded gamma rings}}
\end{center}
\begin{center}
\textbf{Emil Ili\'{c}-Georgijevi\'{c}}
\end{center}
\begin{abstract}
\noindent We introduce graded gamma rings from a more general
point of view via methods developed by Krasner and Halberstadt for
graded rings. We propose three equivalent aspects of studying
graded gamma rings, nonhomogeneous, semihomogeneous and
homogeneous. The graded Jacobson radical of a graded gamma ring is
introduced and its elementwise description is given. Also, a
relation between the graded Jacobson radical and the Jacobson
radical of a graded gamma ring is examined.
\end{abstract}
\noindent\emph{Keywords}: Graded gamma rings and modules, Gamma
anneids and moduloids, Jacobson radical
\\ 2010 \emph{Mathematics Subject Classification} Primary 16Y99, Secondary 16W50, 16N20
\section{Introduction}\label{intro}
\noindent Gamma rings were introduced by Nobusawa \cite{nob} as an
algebraic tool for observing the relationship between the groups
of homomorphisms $\hom(B,C)$ and $\hom(C,B)$ of commutative groups
$B$ and $C.$ Let us recall the notion of a $\Gamma$-ring and a
$\Gamma$-ring in the sense of Nobusawa, following \cite{wb}.\\
If $R$ and $\Gamma$ are abelian additive groups, then $R$ is
called a \emph{$\Gamma$-ring} if, for all $x, y, z\in R$ and
$\alpha, \beta\in\Gamma,$ the following conditions are satisfied:
\begin{itemize}
    \item[$i)$] $x\alpha y\in R;$
    \item[$ii)$] $(x+y)\alpha z=x\alpha z+y\alpha z,$ $x(\alpha+\beta)z=x\alpha z+x\beta
    z,$ $x\alpha(y+z)=x\alpha y+x\alpha z;$
    \item[$iii)$] $(x\alpha y)\beta z=x\alpha(y\beta)z.$
\end{itemize}
It is called a \emph{$\Gamma$-ring in the sense of Nobusawa} if
moreover we have:
\begin{itemize}
    \item[$i')$] $\alpha x\beta\in\Gamma;$
    \item[$ii')$] $(x\alpha y)\beta z=x(\alpha y\beta)z=x\alpha(y\beta)z;$
    \item[$iii')$] if $x\alpha y=0$ for all $x,y\in R,$ then $\alpha=0.$
\end{itemize}
Later on, many mathematicians gave their
contributions to the theory of gamma rings, and mainly concerning
radical theory (see e.g. \cite{bg, cl, rs, ks, dw} and references
therein). In \cite{flin} a group graded gamma ring is defined to
be a $\Gamma$-ring $R$ which is the direct sum of additive
subgroups $R_g,$ $g\in K,$ such that $R_g\Gamma R_h\subseteq
R_{gh},$ for all $g, h\in K,$ where $K$ is a group. However, the
following example motivates us to generalize the previous notion.
Let us suppose we have graded $S$-modules $B$ and $C,$
$B=\bigoplus_{x\in K}B_x,$ $C=\bigoplus_{x\in K}C_x,$ where $S$ is
a $K$-graded ring, $K$ a group and let
\[R=\mathrm{HOM}_S(B,C)=\bigoplus_{x\in K}\mathrm{HOM}_S(B,C)_x,\]
where $\mathrm{HOM}_S(B,C)_x$ is a group of graded morphisms of
degree $x$ \cite{naoy}, that is, of $S$-linear mappings
$\varphi:B\to C$ such that $\varphi(B_g)\subseteq C_{xg},$ for all
$g\in K.$ Similarly, observe $\Gamma=\mathrm{HOM}_S(C,B).$
Clearly, $R$ is a $\Gamma$-ring. If $\varphi_1\in R_x,$
$\varphi_2\in R_z$ and $\psi\in \Gamma_y,$ then
$\varphi_1\psi\varphi_2(B_g)\subseteq
\varphi_1(\psi(C_{zg}))\subseteq \varphi_1(B_{yzg})\subseteq
C_{xyzg},$ and hence, $\varphi_1\psi\varphi_2\in R_{xyz}.$ So, one
way to extend the notion of a group graded gamma ring is to
require
\begin{equation}\label{mu}
R_g\Gamma_fR_h\subseteq R_{gfh},
\end{equation}
for all $f,g,h\in K,$ if $\Gamma=\bigoplus_{g\in K}\Gamma_g.$
However, if we replace \eqref{mu} with
\begin{equation}\label{mu1}
(\exists k\in K)\ R_g\Gamma_fR_h\subseteq R_k,
\end{equation}
for all $f,g,h\in K,$ the additive gradings $R=\bigoplus_{g\in
K}R_g$ and $\Gamma=\bigoplus_{g\in K}\Gamma_g$ of $R$ and
$\Gamma,$ respectively, and the structure of a $\Gamma$-ring $R$
will imply operation (generally partial) in $K.$ Indeed, if
$R_g\Gamma_fR_h\neq\{0\},$ then $k\in K,$ for which
$R_g\Gamma_fR_h\subseteq R_k,$ is unique, and is arbitrary if
$R_g\Gamma_fR_h=\{0\}.$ So, if $R_g\Gamma_fR_h\neq\{0\},$ we may
define $gfh:=k.$ Of course, for those $g,f,h\in K$ for which
$R_g\Gamma_fR_h=\{0\},$ this ternary operation may be defined
arbitrarily in order to make it defined everywhere on $K.$ So,
there is no need to assume anything of $K$ except of being a
nonempty set. We do not stop here, but go further by assuming
different grading sets of $R$ and $\Gamma,$ namely
$R=\bigoplus_{g\in K}R_g$ and $\Gamma=\bigoplus_{h\in H}\Gamma_h.$
In that case, condition
\begin{equation}\label{mu2}
(\forall g, g'\in K)(\forall h\in H)(\exists k\in K)\
R_g\Gamma_hR_{g'}\subseteq R_k
\end{equation}
will make $K$ a partial $\Gamma$-ring, where $\Gamma=H,$ which of
course, can be extended to a $\Gamma$-ring arbitrarily.\\
Let us call the unique $\sigma\in K,$ for which $0\neq
x\in\bigcup_{g\in K}R_g$ belongs to $R_\sigma,$ the \emph{degree}
of $x$ and denote it by $\delta(x).$ The degree of $0\neq
\gamma\in\bigcup_{h\in H}\Gamma_h$ will be denoted by $d(\gamma).$
Sets $A=\bigcup_{g\in K}R_g$ and $G=\bigcup_{h\in H}\Gamma_h$ are
called \emph{homogeneous parts} of $R$ and $\Gamma,$ respectively.
Now, \eqref{mu2} may be interpreted as $\delta(x\gamma y)=ghg',$
where $\delta(x)=g,$ $d(\gamma)=h,$ $\delta(y)=g'$ if
$R_g\Gamma_hR_{g'}\neq\{0\},$ since, we will prove that
$\delta(x\gamma y)$ depends only on $\delta(x),$ $d(\gamma)$ and
$\delta(y)$ if $x\gamma y\neq0.$\\
This discussion is inspired by the similar observations made for
graded groups, rings and modules from \cite{agg} (see also
\cite{cha,ha}).\\
There are many important examples of graded
gamma rings in our sense which are not necessarily group graded
gamma rings. Among them are the following ones.\\
1. Set of rectangular matrices over a division ring.\\
Let $D$ be a division ring and $M=M_{m,n}(D)$ the set of all
matrices of type $m\times n.$ $M$ is a commutative group with
respect to addition of matrices. We know that $M$ is a
$\Gamma$-ring with $\Gamma=M$ if for $x,y,\alpha\in M$ we define
$x\alpha y=x\alpha^Ty,$ where $\alpha^T$ is the transpose of a
matrix $\alpha$ \cite{nob,rs}. However, $M$ is also a graded
$\Gamma$-ring in our sense. Indeed, $M=\bigoplus_{(i,j)\in
I_m\times I_n} M_{i,j},$ $I_m=\{1,\dots,m\},$ $I_n=\{1,\dots,n\},$
where $M_{i,j}$ is the set of matrices over $D$ whose
$(i,j)$-entry is from $D$ and all other entries are zero. Also,
for all pairs $(i,j),$ $(k,l)$ and $(p,q)$ we have
$M_{i,j}M_{k,l}^TM_{p,q}\subseteq M_{r,s},$ for some pair $(r,s).$\\
2. Generalized matrix rings.\\
Let $R,$ $S$ be rings and $V,$ $W$ be an $(R,S)$-bimodule and a
$(S,R)$-bimodule, respectively. Also, let us assume that products
$V\times W$ to $R$ and $W\times V$ to $S,$ denoted by $vw$ and
$wv$ for all $v\in V$ and $w\in W,$ are defined in such a way that
the system $M=(R,V,W,S)$ with these products forms a
Morita context.\\
Let $R(M)=\left(%
\begin{array}{cc}
  R & V \\
  W & S \\
\end{array}%
\right)$ be the set of all $2\times2$-matrices of the form $\left(%
\begin{array}{cc}
  r & v \\
  w & s \\
\end{array}%
\right),$ where $r\in R,$ $s\in S,$ $v\in V,$ $w\in W.$\\
The set $R(M)$ forms a ring under the usual definitions of matrix
addition and multiplication, which is called \emph{generalized
matrix ring} determined by the Morita context $M.$\\
Note that $R(M)$ is a graded $\Gamma$-ring in our sense if we put $\Gamma=\left(%
\begin{array}{cc}
  R & 0 \\
  0 & 0 \\
\end{array}%
\right)\oplus\left(%
\begin{array}{cc}
  0 & 0 \\
  0 & S \\
\end{array}%
\right).$
Indeed, if $R_1=\left(%
\begin{array}{cc}
  R & 0 \\
  0 & 0 \\
\end{array}%
\right),$
$R_2=\left(%
\begin{array}{cc}
  0 & V \\
  0 & 0 \\
\end{array}%
\right),$
$R_3=\left(%
\begin{array}{cc}
  0 & 0 \\
  W & 0 \\
\end{array}%
\right),$
$R_4=\left(%
\begin{array}{cc}
  0 & 0 \\
  0 & S \\
\end{array}%
\right),$ then $R(M)=R_1\oplus R_2\oplus R_3\oplus R_4$ and
$\Gamma=R_1\oplus R_4,$ and for all $i,j\in\{1,2,3,4\}$ and
$k\in\{1,4\}$ there exists $l\in\{1,2,3,4\}$ such that
$R_iR_kR_j\subseteq R_l.$\\
3. Semidirect sums of rings.\\
If $R$ is a ring which contains a subring $S$ and ideal $I$ such
that $R=S\oplus I,$ then $R$ is called the
\emph{semidirect sum} of $S$ and $I$ (see e.g. \cite{ak}).\\
Let $R_1=S$ and $R_2=I.$ Then, by putting $\Gamma=S,$ we get that
$R=R_1\oplus R_2$ is a graded $\Gamma$-ring since $RSR\subseteq
R,$ $R_2\Gamma R_2=ISI\subseteq I=R_2,$ $R_2\Gamma
R_1=ISS\subseteq I=R_2,$ $R_1\Gamma
R_2=SSI\subseteq I=R_2,$ $R_1\Gamma R_1=SSS\subseteq S=R_1.$\\
In particular, the Dorroh extension may be regarded as a graded
gamma ring in our sense.\\
Theory we are about to establish covers the theory of ordinary
$\Gamma$-rings and of group graded $\Gamma$-rings. Also, as
$\Gamma$-rings generalize the concept of a ring, graded
$\Gamma$-rings will provide a generalization of a graded ring.
Before observing graded $\Gamma$-rings, we deal with preliminaries
regarding Krasner's definition of a grading for groups, rings and
modules \cite{cha,ha,agg}. Kelarev \cite{ak1} and Kelarev and
Plant \cite{ak2} also studied this approach to gradings of rings,
namely, $S$-graded rings inducing $S.$ If $R$ is a graded
$\Gamma$-ring, $A$ the homogeneous part of $R$ and $G$ the
homogeneous part of $\Gamma,$ we prove that a $\Gamma$-ring $R$ is
determined by $A$ and $G.$ We also define and give an elementwise
description of the graded Jacobson radical of a graded
$\Gamma$-ring $R$ in terms of $G$ and $A.$ Finally, the comparison
between graded Jacobson and Jacobson radicals of graded
$\Gamma$-rings is given.
\section{Preliminaries}
\noindent Our notion of a graded gamma ring is inspired by the
notion of a grading due to Krasner, the origin of which goes back
to \cite{ug}. This notion in case of rings coincides with the
notion of an $S$-graded ring inducing $S,$ studied by Kelarev in
\cite{ak1} and Kelarev and Plant in \cite{ak2}. In this section,
we gather up all notions and results due to Krasner which are
essential for the rest of the article. More on this theory can be
found in \cite{agg} and references therein (in particular,
\cite{cha} and \cite{ha}) or \cite{sp,mv}. Throughout this
section, $G$ denotes a multiplicative group with the neutral
element $1$ and $\Delta$ is a nonempty set. Even when the
operation is denoted multiplicatively, we will use $\bigoplus$ to
denote the (restrictive) direct product of groups as if it were a
(restrictive) direct sum in question.
\begin{definition}[\cite{agg}]
Every mapping
\begin{equation}\label{1g1}
\gamma:\Delta\to\mathrm{Sg}(G),\ \gamma(\delta)=G_\delta\
(\delta\in\Delta),
\end{equation} such that
$G=\bigoplus_{\delta\in\Delta}G_\delta,$ where $\mathrm{Sg}(G)$ is
the set of all subgroups of $G,$ is called a \emph{grading}. A
group with a grading is called \emph{a graded group}. A grading is
called \emph{strict} if $G_\delta\neq\{1\},$ for all
$\delta\in\Delta.$ If $\Delta^*=\{\delta\in\Delta\ |\
G_\delta\neq\{1\}\},$ then $\gamma^*=\gamma|_{\Delta^*}$ is a
strict grading of $G$ and is called a \emph{strict kernel} of
grading \eqref{1g1}. Elements $\delta\in\Delta$ are called
\emph{degrees} and the corresponding $G_\delta$ are called
\emph{homogeneous components}. The set
$H=\bigcup_{\delta\in\Delta}G_\delta$
($=\bigcup_{\delta\in\Delta^*}G_\delta$ if $G\neq\{1\}$) is called
the \emph{homogeneous part} of a graded group $G$ and elements
$x\in H$ are called \emph{homogeneous}. The unique
$\xi\in\Delta^*$ for which $1\neq x\in H$ belongs to $G_\xi$ is
called the \emph{degree of $x$} and is denoted by $\delta(x).$
\end{definition}
\noindent Element $1$ generally speaking does not have a degree
and $\delta\in\Delta\setminus\Delta^*$ are called \emph{empty
degrees} \cite{agg}. However, we may associate a degree from
$\Delta\setminus\Delta^*$ to $1,$ which we denote by $0$ and call
it a \emph{zero degree} \cite{agg}. If $\Delta=\Delta^*\cup\{0\},$
and if we put $\delta(1)=0,$ the grading is called \emph{proper}
\cite{agg}.
Throughout this paper we assume all gradings to be proper.\\
Let $G$ be a graded group with grading \eqref{1g1} and let
\begin{equation}\label{1g2}
\gamma':\Delta'\to\mathrm{Sg}(G), \gamma'(\delta')=G_{\delta'}\
(\delta'\in\Delta')
\end{equation}
be another grading of $G.$ We say that gradings \eqref{1g1} and
\eqref{1g2} are \emph{equivalent} \cite{agg} if there exists a
bijective mapping $\varphi:\Delta\to\Delta'$ such that
$\gamma(\delta)=\gamma'(\varphi(\delta)),$ for every
$\delta\in\Delta.$ These gradings are \emph{weakly equivalent}
\cite{agg} if their strict kernels are equivalent. It is clear
that strict gradings are determined up to an equivalence by the
set $\{G_\delta\ |\ \delta\in\Delta\},$ while gradings are
determined up to a weak equivalence by $\{G_\delta\ |\
\delta\in\Delta^*\}.$ The homogeneous part $H$ of a graded group
$G,$ together with the group structure of $G,$ determines a
grading of $G$ up to a weak equivalence, for if $1\neq a, b\in H$
and $a\in G_\delta,$ then $b\in G_\delta$ if and only if $ab\in
H,$ which is to say that the set $\{G(a):=\{x\in H\ |\ ax\in H\}\
|\ a\in H^*=H\setminus\{1\}\}$ coincides with the set $\{G_\delta\
|\ \delta\in\Delta^*\}.$ So, it is natural to define a graded
group $G$ from the so-called \emph{semihomogeneous aspect} as an
ordered pair $(G,H),$ where $H\subseteq G$ is the homogeneous part
with respect to some grading of $G,$ see \cite{agg}. However, in
order to do that, the following characterization of that
homogeneous part is given.
\begin{theorem}[\cite{agg}]\label{agg}
A nonempty subset $H$ of a group $G$ is the homogeneous part of
$G$ with respect to some grading of $G$ if and only if the
following conditions are satisfied:
\begin{itemize}
    \item[$i)$] $1\in H;$
    \item[$ii)$] $x\in H\Rightarrow x^{-1}\in H;$
    \item[$iii)$] $x,y,z,xy,yz\in H\wedge y\neq 1\Rightarrow xz\in H;$
    \item[$iv)$] $x,y\in H\wedge xy\notin H\Rightarrow xy=yx;$
    \item[$v)$] $H$ generates $G;$
    \item[$vi)$] If $n\geq2$ and if the elements $x_1,\dots,x_n\in H$
    are such that for all $i,j\in\{1,\dots,n\},$ $i\neq j,$ $x_ix_j\notin
    H,$ then $x_1\dots x_n\neq 1.$
\end{itemize}
\end{theorem}
\noindent Multiplicative operation on $G$ induces a partial
operation in $H.$ Namely, if $x,y\in H,$ then $xy$ is
def\mbox{}ined in $H$ if and only if $xy\in G$ is the element from
$H,$ and in that case the result is the same and we write it the
same way. If this situation occurs, we say that elements $x,y$ are
\emph{composable} (\emph{addable} in case of an additive
operation) and we write $x\# y.$ Clearly, $x\# y$ if and only if
$x,y$ come from the same subgroup $G(a)=\{ x\in H\ |\ ax\in H
\},$ $a\in H^*,$ see \cite{agg}.\\
In case when $H$ with the induced operation from $G$ is given, we
may reconstruct $G$ up to $H$-isomorphism. Indeed, if $a\in H^*,$
then $G(a)$ may be def\mbox{}ined as $G(a)=\{ x\in H\ |\ a\# x
\}.$ $G$ is then the direct sum of different subgroups $G(a)$ and
$H$ is obviously the homogeneous part of $G.$ The group $G$ which
is obtained in this way is called \emph{the linearization} of the
structure $(H,\cdot)$ and we denote it by $\overline{H},$ see
\cite{agg}. Hence, $\overline{H}=\bigoplus_{a\in H^*}G(a).$\\
There is a natural idea now to def\mbox{}ine the graded group from
the so called \emph{homogeneous aspect} using the corresponding
homogeneous part, at least up to isomorphism, see \cite{agg}. We
need to characterize the structure $(H,\cdot),$ which is the
homogeneous part of some graded group, with operation induced from
that group. This characterization is the subject of the following
theorem.
\begin{theorem}[\cite{agg}]\label{homog}
The structure $(H,\cdot)$ is the homogeneous part of some graded
group $G,$ with operation induced from that group, if and only if
the following conditions hold:
\begin{itemize}
    \item[$i)$] $(\exists 1\in H)(\forall x\in H)\ x\#1\wedge x1=x;$
    \item[$ii)$] $(\forall x\in H)\ x\#x;$
    \item[$iii)$] $(\forall x,y,z\in H)\ x\#y\wedge y\#z\wedge y\neq 1\Rightarrow x\#z;$
    \item[$iv)$] For all $a\in H$ for which $a\neq 1,$ $H(a)=\{ x\in H\ |\ a\#x
    \}$ is a group, called an \emph{addibility group of $H.$}
\end{itemize}
\end{theorem}
\begin{definition}[\cite{agg}]\label{homogroupoid}
The structure $(H,\cdot)$ which satisfies conditions of Theorem
\ref{homog} is called a \emph{homogroupoid}.
\end{definition}
\noindent Let $H$ be a homogroupoid and $\Delta^*$ the set of
groups $H(a)=\{x\in H\ |\ ax\in H\},$ $a\in H^*=H\setminus\{1\},$
and let us denote elements of $\Delta^*$ by
$\delta,\xi,\eta,\dots.$ If $\delta\in\Delta^*,$ we denote by
$H(\delta)$ or by $G_\delta$ or by $H_\delta$ the corresponding
group $H(a)$ which defines $\delta.$ Hence, if $a\in H^*,$ then we
define its degree as $\delta(a):=\xi$ if $a\in H(\xi)=G_\xi,$ and
so, writing $H(a)$ has the same meaning as writing $H(\delta(a)),$
see \cite{cha,ha,agg}. Thus, the linearization $\overline{H}$ of
$H$ is
$\bigoplus_{\delta\in\Delta^*}H_\delta=\bigoplus_{\delta\in\Delta}H_\delta$
if we put $\Delta=\Delta^*\cup\{0\},$ where $0$ corresponds to a
trivial subgroup $\{1\};$ $\delta(1)=0.$ These situations will
occur in the sequel, and after this discussion, these should not
cause any confusion.
\begin{definition}[\cite{agg}]
A subgroup $S$ of a graded group
$G=\bigoplus_{\delta\in\Delta}G_\delta$ is called a
\emph{homogeneous subgroup} of $G$ if
$S=\bigoplus_{\delta\in\Delta}S\cap G_\delta.$
\end{definition}
\begin{definition}[\cite{agg}]
A nonempty subset $K$ of a homogroupoid $H$ is called a
\emph{subhomogroupoid} if $K$ is the homogeneous part of a
homogeneous subgroup of $\overline{H}.$
\end{definition}
\noindent A subhomogroupoid $K$ of a homogroupoid $H$ is normal if
each addibility group $K(a)$ of $K$ is a normal subgroup of
$H(a),$ where $a\in H^*.$ If $K$ is a normal subhomogroupoid of
$H,$ then the homogeneous part of $\overline{H}/\overline{K}$ is
$\bigcup_{a\in H^*}H(a)/H(a)\cap K$ \cite{agg}.
\begin{definition}[\cite{agg}]
A homomorphism $f:G\to G'$ of graded groups with homogeneous parts
$H$ and $H',$ respectively, is called \emph{quasihomogeneous} if
$f(H)\subseteq H'.$ A quasihomogeneous homomorphism $f$ is called
a \emph{homogeneous homomorphism} if for $x, y\in H$ with
$f(x)\neq1',$ $f(y)\neq1',$ we have
\[\delta(f(x))=\delta(f(y))\Rightarrow\delta(x)=\delta(y).\]
\end{definition}
\noindent The homogeneous counterpart of the above definition
follows.
\begin{definition}[\cite{agg}]
Let $H$ and $H'$ be homogroupoids. The mapping $f:H\to H'$ is
called \emph{quasihomomorphism} if $x\#y$ implies $f(x)\#f(y)$ and
$f(xy)=f(x)f(y),$ $x,y\in H.$ A quasihomomorphism is called
\emph{homomorphism} if the composability of nonidentity images
implies the composability of the originals.
\end{definition}
\noindent As we have seen, homogroupoids are homogeneous parts of
graded groups with induced partial operation. Analogously, an
\emph{anneid} and a \emph{moduloid} represent homogeneous parts of
a graded ring and a graded module, respectively, with induced
operations \cite{cha,ha,agg}. A ring $R$ is called \emph{graded}
\cite{cha,ha,agg,ak2,ak1} if $(R,+)$ is a commutative graded group
in the above sense with grading
$\gamma_R:\Delta\to\mathrm{Sg}(R,+),$ $\gamma_R(\delta)=R_\delta,$
and if for all $\xi,\eta\in\Delta,$ $R_\xi R_\eta\subseteq
R_\zeta,$ for some $\zeta\in\Delta,$ while a right $R$-module $M$
is \emph{graded} \cite{cha,ha,agg} if
$R=\bigoplus_{\delta\in\Delta}R_\delta$ is a graded ring and $M$ a
commutative graded group with grading
$\gamma_M:D\to\mathrm{Sg}(M,+),$ $\gamma_M(d)=M_d,$ and if for all
$\delta\in\Delta$ and $d\in D$ there exists $t\in D$ such that
$M_dR_\delta\subseteq M_t.$ As an example of a graded ring one can
take, for instance, every group graded ring in the usual sense.
However, there are important types of rings which are graded in
Krasner's sense, but not necessarily group graded,
such as the already mentioned semidirect sums of rings (see also \cite{ak}).\\
If $M$ and $M'$ are $A$-moduloids, than the mapping $f:M\to M'$ is
called a \emph{quasihomomorphism} \cite{cha,ha,agg} if for all
$x,y\in M, a\in A,$ $f(xa)=f(x)a$ and if $x\#y$ implies
$f(x)\#f(y)$ and $f(x+y)=f(x)+f(y).$ A quasihomomorphism is called
a \emph{homomorphism} \cite{cha,ha,agg} if moreover we have that
the addibility of nonzero images implies the addibility of
originals. A right $A$-moduloid $M$ is called \emph{right regular}
\cite{cha,ha,agg} if the mapping $a\to xa$ $(x\in M)$ from $A$ to
$M$ is a homomorphism. An $A$-moduloid $M$ is called \emph{small}
\cite{ha} if it is regular and if there exists $w$ such that
$M=M(w)A.$
\section{Graded $\Gamma$-rings}
\noindent We study three equivalent aspects of graded
$\Gamma$-rings, nonhomogeneous, semihomogeneous and homogeneous,
by analogy with \cite{agg} for graded groups, rings and modules.
\subsection{Nonhomogeneous aspect}
\noindent Let $R$ be a $\Gamma$-ring in the sense of Nobusawa and
let $R$ and $\Gamma$ be graded groups in the sense of Krasner with
gradings $\gamma_R:\Delta\to\mathrm{Sg}(R,+),$
$\gamma_R(\delta)=R_\delta,$
$R=\bigoplus_{\delta\in\Delta}R_\delta,$ and
$\gamma_\Gamma:D\to\mathrm{Sg}(\Gamma,+),$
$\gamma_\Gamma(d)=\Gamma_d,$ $\Gamma=\bigoplus_{d\in D}\Gamma_d,$
respectively. Also, let $A$ be the homogeneous part of $R$ and $G$
the homogeneous part of $\Gamma.$ If the following conditions are
satisfied:
\begin{equation}\label{1}
(\forall\xi,\eta\in\Delta)(\forall d\in D)(\exists\zeta\in\Delta)\
R_\xi\Gamma_dR_\eta\subseteq R_\zeta,
\end{equation}
\begin{equation}\label{2}
(\forall s,t\in D)(\forall\delta\in\Delta)(\exists d\in D)\
\Gamma_sR_\delta\Gamma_t\subseteq\Gamma_d,
\end{equation}
then $R$ is called a \emph{graded $\Gamma$-ring of Nobusawa of
type $(\Delta,D),$} where $\Delta$ and $D$ are sets. $R$ is called
a \emph{graded $\Gamma$-ring of type $(\Delta,D)$} if $R$ is
observed as a $\Gamma$-ring and if only \eqref{1} is satisfied.\\
The following lemma represents an analogous result obtained for
graded rings, see, for example, \cite{agg}.
\begin{lemma}\label{lemma1}
The condition \eqref{1} is satisfied if and only if the following
conditions hold:
\begin{itemize}
    \item[$i)$] $AGA\subseteq A;$
    \item[$ii)$] If $x,y\in A$ and $\alpha\in G$ with $x\alpha
    y\neq0,$ then $\delta(x\alpha y)$ depends only on $\delta(x),$
    $\delta(y)$ and $d(\alpha),$
\end{itemize}
where $\delta(x)\in\Delta$ and $d(\alpha)\in D$ denote degrees of
$x\in A$ and $\alpha\in G,$ respectively.
\end{lemma}
\begin{proof}
Clearly, \eqref{1} follows from $i)$ and $ii).$ Conversely,
suppose that \eqref{1} holds and let us prove that $AGA\subseteq
A.$ Let $x,y\in A$ and $\alpha\in G.$ Then there exist
$\xi,\eta\in\Delta$ and $d\in D$ such that $x\in R_\xi,$ $y\in
R_\eta$ and $\alpha\in\Gamma_d$ which implies $x\alpha y\in
R_\xi\Gamma_dR_\eta,$ and by \eqref{1} there exists
$\zeta\in\Delta$ such that $R_\xi\Gamma_dR_\eta\subseteq
R_\zeta\subseteq A,$ and so $AGA\subseteq A$ and $i)$ holds.
Actually, $i)$ implies $ii).$ Indeed, assume that $i)$ holds and
let $x,y,z,u\in A,$ $\alpha\in G,$ $\delta(x)=\delta(z),$
$\delta(y)=\delta(u),$ $x\alpha y\neq0,$ $z\alpha y\neq0,$
$x\alpha u\neq0,$ $z\alpha u\neq0.$ Then $x+z\in A,$ $y+u\in A,$
and hence, $(x+z)\alpha y\in A,$ that is, $x\alpha y+z\alpha y\in
A,$ that is $\delta(x\alpha y)=\delta(z\alpha y).$ Similarly,
$\delta(x\alpha y)=\delta(x\alpha u).$ If $x,x',y,y'\in A,$
$\alpha\in G,$ $\delta(x)=\delta(x'),$ $\delta(y)=\delta(y'),$
$x\alpha y\neq0,$ $x'\alpha y'\neq0,$ then $\delta(x\alpha
y)=\delta(x'\alpha y').$ Indeed, if $x\alpha y'\neq0,$ then,
according to what we have already proved, $\delta(x\alpha
y)=\delta(x\alpha y')$ and $\delta(x\alpha y')=\delta(x'\alpha
y').$ Analogously, $\delta(x\alpha y)=\delta(x'\alpha y')$ if
$x'\alpha y\neq0.$ If however $x\alpha y'=0$ and $x'\alpha y=0,$
then $x\alpha y+x'\alpha y'=(x+x')\alpha(y+y')\in A,$ and so
$\delta(x\alpha y)=\delta(x'\alpha y').$ If $x,x',y,y'\in A,$
$\alpha,\beta\in G,$ $\delta(x)=\delta(x'),$
$\delta(y)=\delta(y'),$ $x\alpha y\neq0,$ $x'\alpha y'\neq0,$
$x\beta y\neq0,$ $x'\beta y'\neq0,$ and $d(\alpha)=d(\beta),$ then
$x\alpha y+x\beta y=x(\alpha+\beta)y\in A,$ and so $\delta(x\alpha
y)=\delta(x\beta y).$ Since, by what we have already proved,
$\delta(x\beta y)=\delta(x'\beta y')$ and $\delta(x\alpha
y)=\delta(x'\alpha y'),$ we have $\delta(x\alpha y)=\delta(x\beta
y)=\delta(x'\alpha y')=\delta(x'\beta y'),$ and so $ii)$ holds.
\end{proof}
\noindent The following lemma can be proved similarly.
\begin{lemma}\label{lemma2}
The condition \eqref{2} is satisfied if and only if the following
conditions hold:
\begin{itemize}
    \item[$i)$] $GAG\subseteq G;$
    \item[$ii)$] If $\alpha, \beta\in G$ and $x\in A$ with $\alpha
    x\beta\neq0,$ then $d(\alpha x\beta)$ depends only on
    $d(\alpha),$ $d(\beta)$ and $\delta(x),$
\end{itemize}
where $\delta(x)\in\Delta$ and $d(\alpha)\in D$ denote degrees of
$x\in A$ and $\alpha\in G,$ respectively.
\end{lemma}
\noindent These lemmas allow us to define ternary operations
$\Delta\times D\times\Delta\to\Delta$ and $D\times\Delta\times
D\to D.$ Let $R$ be a graded $\Gamma$-ring of Nobusawa of type
$(\Delta,D).$ Let us prove that $\zeta$ from condition \eqref{1}
is unique if $R_\xi\Gamma_dR_\eta\neq\{0\}.$ Suppose that there is
another $\zeta'\in\Delta$ such that $R_\xi\Gamma_dR_\eta\subseteq
R_{\zeta'}.$ Then $R_\xi\Gamma_dR_\eta\subseteq R_\zeta\cap
R_{\zeta'}=\{0\}$ which implies $R_\xi\Gamma_dR_\eta=\{0\},$
contrary to assumption. Hence $\zeta$ from \eqref{1} is unique. A
similar argument shows that $d\in D$ from \eqref{2} is unique if
$\Gamma_sR_\delta\Gamma_t\neq\{0\}.$ Thus for $\xi, \eta\in\Delta$
and $d\in D$ such that $R_\xi\Gamma_dR_\eta\neq\{0\}$ we may
define $\xi d\eta\in\Delta$ to be that unique $\zeta$ for which
$R_\xi\Gamma_dR_\eta\subseteq R_\zeta.$ In that case we say that
$\xi, d$ and $\eta$ are \emph{composable}. However, since we
assume that the gradings in hand are proper, we also define $\xi
d\eta=0$ if $R_\xi\Gamma_dR_\eta=\{0\}.$ We similarly define the
composability of $s, \delta$ and $t.$
\begin{example}\label{example}
1. Let $R$ be a graded ring in Krasner's sense \cite{cha,ha,agg},
i.e., $R$ is the direct sum of its additive subgroups $R_\delta,$
where $\delta$ runs through a nonempty set $\Delta,$ and for all
$\xi,\eta\in\Delta$ there exists $\zeta\in\Delta$ such that $R_\xi
R_\eta\subseteq R_\zeta$ (see also \cite{ak2}). Then $R$ is a
graded $\Gamma$-ring if we put $\Gamma=R.$ Indeed, let $\xi,$
$\eta,$ $\zeta\in\Delta.$ Then, if $R_\xi R_\eta
R_\zeta\neq\{0\},$ we have \[R_\xi R_\eta R_\zeta=(R_\xi
R_\eta)R_\zeta\subseteq R_\lambda R_\zeta\subseteq R_\delta,\] for
some $\lambda,$ $\delta\in\Delta$ according to the definition of a
graded ring. If, however, $R_\xi R_\eta R_\zeta=\{0\},$ then
$R_\xi R_\eta R_\zeta\subseteq R_\delta,$ for
every $\delta\in\Delta.$\\
2. Let $D$ be a division ring and $M=M_{m,n}(D)$ the group of
rectangular matrices of type $m\times n$ over $D.$ Then
$M=\bigoplus_{(i,j)\in I_m\times I_n}M_{i,j},$ where $M_{i,j}$
denotes the set of matrices whose $(i,j)$-entry is from $D$ and
all other entries are zero, $I_m=\{1,\dots,m\},$
$I_n=\{1,\dots,n\}.$ Recall from Section \ref{intro}, $M$ is a
graded $\Gamma=M$-ring, if for $x,y,\alpha\in M$ we define
$x\alpha y=x\alpha^Ty,$ where $\alpha^T$ denotes the transpose of
$\alpha.$ Notice that it is also a graded $\Gamma$-ring in the
sense of Nobusawa. If $D$ is a graded division ring in Krasner's
sense \cite{agg} with the grading set $\Delta,$ i.e., it is a
graded ring in Krasner's sense and every nonzero homogeneous
element is invertible, then, for each pair $(i,j)\in I_m\times
I_n,$ $M_{i,j}=\bigoplus_{\delta\in\Delta}(M_{i,j})_\delta,$ where
$(M_{i,j})_\delta$ is the set of matrices from $M_{i,j}$ whose
$(i,j)$-entry comes from $D_\delta,$ which is obviously a subgroup
of $M_{i,j}.$ Then again $M$ represents a graded $\Gamma=M$-ring
in the sense of Nobusawa since for all $(i,j),(k,l),(p,q)\in
I_m\times I_n$ and $\delta,\xi,\eta\in\Delta$ there exist
$(r,s)\in I_m\times I_n$ and $\zeta\in\Delta$ such that
$(M_{i,j})_\delta(M_{k,l})_\xi^T(M_{p,q})_\eta\subseteq(M_{r,s})_\zeta.$
\end{example}
\subsection{Semihomogeneous aspect}
\noindent Let $R$ be a graded $\Gamma$-ring of Nobusawa of type
$(\Delta,D).$ Then $A$ and $G$ satisfy known conditions for
homogeneous parts of graded groups (see Theorem \ref{agg}).
According to lemmas \ref{lemma1} and \ref{lemma2}, there are two
more conditions, namely, $AGA\subseteq A$ and $GAG\subseteq G.$
Conversely, if all of these conditions for $A\subseteq R$ and
$G\subseteq\Gamma$ are satisfied, then we know that $A$ and $G$
are homogeneous parts of graded groups $R$ and $\Gamma$ of type
$\Delta$ and $D,$ respectively, with gradings
$\gamma_R:A^*\to\mathrm{Sg}(R,+),$ $\gamma_R(a)=A(a)=\{x\in A\ |\
a+x\in A\},$ $\gamma_\Gamma:G^*\to\mathrm{Sg}(\Gamma,+),$
$\gamma_\Gamma(\alpha)=G(\alpha)=\{\xi\in G\ |\ \alpha+\xi\in
G\},$ and by lemmas \ref{lemma1} and \ref{lemma2}, $R$ is a graded
$\Gamma$-ring of Nobusawa of type $(\Delta,D).$ Hence, we may
define a graded $\Gamma$-ring of Nobusawa from the semihomogeneous
point of view as a quadruple $(R,\Gamma,A,G),$ where $R$ is a
$\Gamma$-ring of Nobusawa, and $A\subseteq R,$ $G\subseteq\Gamma$
nonempty subsets such that:
\begin{itemize}
    \item[$i)$] $0\in A;$
    \item[$ii)$] $x\in A\Rightarrow -x\in A;$
    \item[$iii)$] $x,y,z,x+y,y+z\in A\wedge y\neq0\Rightarrow x+z\in A;$
    \item[$iv)$] $A$ generates $R;$
    \item[$v)$] If $n\geq2$ and if elements $x_1,\dots,x_n\in A$ are
    such that for all $i,j\in\{1,\dots,n\},$ $i\neq j,$ $x_i+x_j\notin
    A,$ then $x_1+\dots+x_n\neq0;$
    \item[$vi)$] $AGA\subseteq A;$
    \item[$i')$] $0\in G;$
    \item[$ii')$] $\alpha\in G\Rightarrow -\alpha\in G;$
    \item[$iii')$] $\alpha,\beta,\gamma,\alpha+\beta,\beta+\gamma\in G\wedge \beta\neq0\Rightarrow \alpha+\gamma\in G;$
    \item[$iv')$] $G$ generates $\Gamma;$
    \item[$v')$] If $n\geq2$ and if elements $\alpha_1,\dots,\alpha_n\in G$ are
    such that for all $i,j\in\{1,\dots,n\},$ $i\neq j,$ $\alpha_i+\alpha_j\notin
    G,$ then $\alpha_1+\dots+\alpha_n\neq0;$
    \item[$vi')$] $GAG\subseteq G.$
\end{itemize}
\begin{remark}
Here, we denoted neutral elements of $R$ and $\Gamma$ by the same
symbol, $0,$ but there is no fear of confusion.
\end{remark}
\subsection{Homogeneous aspect}
\noindent Let $R$ be a graded $\Gamma$-ring of Nobusawa, $A$ the
homogeneous part of $R,$ $G$ the homogeneous part of $\Gamma.$
Observe the restriction of addition of $R$ on $A,$ and similarly
of addition of $\Gamma$ on $G.$ Also, let us restrict operations
$R\times\Gamma\times R\to R$ and $\Gamma\times
R\times\Gamma\to\Gamma$ on $A\times G\times A\to A$ and $G\times
A\times G\to G,$ respectively. Then restricted additions on $A$
and $G$ are partial operations since sum of two nonzero
homogeneous elements does not have to be a homogeneous element. On
the other hand, according to lemmas \ref{lemma1} and \ref{lemma2},
ternary operations $A\times G\times A\to A$ and $G\times A\times
G\to G$ are everywhere defined. We then call $A$ a
$G$-\emph{anneid of Nobusawa} or a \emph{gamma anneid of
Nobusawa}. If $R$ were only a $\Gamma$-ring, we would obtain a
structure called a $G$-\emph{anneid} or a \emph{gamma anneid}.
Notice that if $A$ is a $G$-anneid, then $A$ and $G$ are
commutative homogroupoids. The set of addibility groups of $A$
will be denoted by $\Delta^*,$ while the set of addibility groups
of $G$ will be denoted by $D^*.$ The corresponding linearizations
$\overline{A}$ and $\overline{G}$ of $A$ and $G,$ respectively,
are hence $\bigoplus_{\delta\in\Delta^*}A_\delta=\bigoplus_{a\in
A^*}A(a)$ and $\bigoplus_{d\in D^*}G_d=\bigoplus_{\alpha\in
G^*}G(\alpha),$ respectively. Naturally, if we put
$\Delta=\Delta^*\cup\{0\}$ and $D=D^*\cup\{0\},$
$\overline{A}=\bigoplus_{\delta\in\Delta}A_\delta$ and
$\overline{G}=\bigoplus_{d\in D}G_d.$
\begin{remark}
Gamma anneid could also be called a \emph{gamma ringoid} but we
kept the term which resembles the original notion of the
homogeneous part of a graded ring in Krasner's sense, the
\emph{anneid}, see \cite{cha,ha,agg}.
\end{remark}
\noindent The following theorem gives us the characterization of
gamma anneids.
\begin{theorem}
Let $A$ and $G$ be commutative homogroupoids. Then $A$ is a
$G$-anneid of Nobusawa if and only if:
\begin{itemize}
    \item[$i)$] $(\forall x,y\in A)(\forall\alpha\in G)\ x\alpha y\in A;$
    \item[$ii)$] $(\forall\alpha,\beta\in G)(\forall x\in A)\ \alpha x\beta\in G;$
    \item[$iii)$] $(\forall a,a',b,b'\in A)(\forall\gamma,\gamma'\in G)\ a\#a'\Rightarrow a\gamma b\#a'\gamma b\wedge(a+a')\gamma b=a\gamma b+a'\gamma
    b,$ $\gamma\#\gamma'\Rightarrow a\gamma b\#a\gamma'b\wedge a(\gamma+\gamma')b=a\gamma
    b+a\gamma'b,$ $b\#b'\Rightarrow a\gamma b\#a\gamma b'\wedge a\gamma(b+b')=a\gamma b+a\gamma b';$
    \item[$iv)$] $(\forall a,b,c\in A)(\forall\gamma,\beta\in G)\ (a\gamma b)\beta c=a\gamma(b\beta c)=a(\gamma b\beta)c\wedge a\gamma b=0\ (a,b\in A)\Rightarrow\gamma=0.$
\end{itemize}
\end{theorem}
\begin{proof}
It is clear that the given conditions are necessary. Now, suppose
that $i)-iv)$ hold. Let $R=\overline{A}=\bigoplus_{a\in A^*}A(a)$
and $\Gamma=\overline{G}=\bigoplus_{\alpha\in G^*}G(\alpha)$ and
let $\bar{x}_i,\bar{y}\in R,$ $i=1,2,$ $\bar{\alpha}\in\Gamma.$
Then $\bar{x}_i,$ $\bar{y}$ and $\bar{\alpha}$ are of the form
$\sum x_\delta^i,$ $\sum y_\delta$ and $\sum\alpha_d,$
respectively, where $x_\delta^i,y_\delta\in A_\delta=\{x\in A\ |\
x=0\vee(x\neq0\wedge\delta(x)=\delta)\},$ $\alpha_d\in
G_d=\{\alpha\in G\ |\ \alpha=0\vee(\alpha\neq0\wedge
d(\alpha)=d)\},$ $x_\delta^i=y_\delta=0$ and $\alpha_d=0$ for all
but finitely many $\delta$ and $d,$ respectively. Now put
$\bar{x}_i\bar{\alpha}\bar{y}:=\sum_\delta\sum_d\sum_{\delta'}x^i_\delta\alpha_dy_{\delta'}.$
Then
\begin{eqnarray}\nonumber
(\bar{x}_1+\bar{x}_2)\bar{\alpha}\bar{y}&=&\sum_\delta(x^1_\delta+x^2_\delta)\sum_d\alpha_d\sum_{\delta'}y_{\delta'}\\
\nonumber&=&\sum_\delta\sum_d\sum_{\delta'}x^1_\delta\alpha_dy_{\delta'}+\sum_\delta\sum_d\sum_{\delta'}x^2_\delta\alpha_dy_{\delta'}\\
\nonumber&=&\bar{x}_1\bar{\alpha}\bar{y}+\bar{x}_2\bar{\alpha}\bar{y}.
\end{eqnarray}
We analogously prove the other properties which make $R$ a graded
$\Gamma$-ring of Nobusawa.
\end{proof}
\begin{remark}
It is clear which of stated conditions are necessary and
sufficient in order for $A$ to be a $G$-anneid.
\end{remark}
\noindent Hence, given a $G$-anneid $A,$ $R=\overline{A},$ and
$\Gamma=\overline{G},$ $R$ is a graded $\Gamma$-ring of type
$(\Delta,D),$ where, $\Delta=\Delta^*\cup\{0\}$ is the grading set
of $R,$ and $D=D^*\cup\{0\}$ the grading set of $\Gamma.$ Let us
recall once again, elements of $\Delta^*$ are denoted by
$\delta,\xi,\eta,\dots,$ but is actually comprised of addibility
groups $A(a)=\{x\in A\ |\ a+x\in A\},$ $a\neq0.$ The \emph{degree}
$\delta(a)$ of $0\neq a\in A$ is defined to be the element
$\delta\in\Delta^*$ which denotes $A(a).$ Conversely, given
$\xi\in\Delta^*,$ the corresponding addibility group will be
denoted by $A(a)$ if $\delta(a)=\xi,$ or it will be denoted by
$A(\xi)$ or by $A_\xi.$ So, when there is no danger of confusion,
the degree of an element and the element will be denoted the same
way. We also keep this convention for the grading set of $\Gamma.$
\begin{example}
Let us take a look at the set $M=M_{m,n}(D)$ of rectangular
matrices of type $m\times n$ over a division ring $D$ once again.
As we saw, it can be regarded as a graded $\Gamma=M$-ring in the
sense of Nobusawa (see Example \ref{example}). Let
$A=\bigcup_{(i,j)\in I_m\times I_n}M_{i,j}.$ Then it is easy to
see that $A$ is a $G=A$-anneid in the sense of Nobusawa if for
$a,b,\alpha\in A$ we define $a\alpha b=a\alpha^Tb.$
\end{example}
\begin{example}
Let $M$ and $M'$ be small moduloids over an anneid, and let
$A=\hom(M,M')$ and $G=\hom(M',M)$ be sets of homomorphisms from
$M$ to $M'$ and from $M'$ to $M,$ respectively. Without the
assumption that $M$ and $M'$ are small, we would not know whether
$A$ and $G$ are homogroupoids. However, since $M$ and $M'$ are
small, both $A$ and $G$ are homogroupoids according to \cite{ha}.
Take $f,g\in A$ and $h\in G.$ Then, for all $x\in M,$ we have
$g(x)\in M',$ $h(g(x))\in M$ and finally, $f(h(g(x)))\in M'.$
Thus, we have $fhg\in A.$ It is easy to verify that $A$ is a
$G$-anneid.
\end{example}
\noindent We have seen three aspects of graded gamma rings in the
sense of Nobusawa and of graded gamma rings which are mutually
equivalent. In the rest of the article we chose to work in the
frame of the homogeneous aspect. It is justified, since in proving
some results concerning homogeneous objects, we need not to think
whether some nonhomogeneous elements are really necessary.
\section{Graded $\Gamma$-modules}
\noindent Let $R$ be a graded $\Gamma$-ring of type $(\Delta,D),$
$V$ a graded additive abelian group of type $\Delta_V$ and let $V$
be a right $R\Gamma$-module (as defined in \cite{rs}). Also, $M$
will be denoting the homogeneous part of $V,$ $A$ the homogeneous
part of $R$ and $G$ the homogeneous part of $\Gamma.$
\begin{definition}
$V$ is called a \emph{graded right $R\Gamma$-module} if
\begin{equation}\label{3}
(\forall \xi\in \Delta_V)(\forall d\in D)(\forall \eta\in
\Delta)(\exists \zeta\in\Delta_V)\ V_\xi\Gamma_d R_\eta\subseteq
V_\zeta.
\end{equation}
\end{definition}
\noindent As in the case of graded $\Gamma$-rings, the following
lemma holds.
\begin{lemma}\label{4}
The condition \eqref{3} is satisfied if and only if the following
conditions hold:
\begin{itemize}
    \item[$i)$] $MGA\subseteq M;$
    \item[$ii)$] If $w\in M,$ $\alpha\in\Gamma,$ $x\in A$ with $w\alpha
    x\neq0,$ then $\delta(w\alpha x)$ depends only on
    $\delta(w),$ $d(\alpha),$ and $\delta(x).$
\end{itemize}
\end{lemma}
\noindent We proceed by giving the notion of a graded gamma module
from the homogeneous aspect.
\begin{definition}
The ordered pair $(M,A),$ where $M$ is an abelian homogroupoid and
$A$ a $G$-anneid, is called an \emph{$AG$-moduloid} (right) if the
mapping $(w,\alpha,x)\to w\alpha x$ $(w\in M, \alpha\in G, x\in
A)$ has the following properties:
\begin{itemize}
    \item[$i)$] $x\#y\Rightarrow w\alpha x\#w\alpha y\wedge w\alpha(x+y)=w\alpha x+w\alpha y;$
    \item[$ii)$] $w\#w'\Rightarrow w\alpha x\#w'\alpha x\wedge (w+w')\alpha x=w\alpha x+w'\alpha x;$
    \item[$iii)$] $w\beta(x\alpha y)=(w\beta x)\alpha y,$
\end{itemize}
for all $w,w'\in M,$ $\alpha,\beta\in G$ and $x,y\in A.$
\end{definition}
\noindent We may define external multiplication of degrees as
follows. If $\xi\in\Delta_V,$ $d\in D$ and $\eta\in\Delta,$ then
$\xi d\eta:=\zeta\in\Delta_V$ if $\{0\}\neq V_\xi\Gamma_d
R_\eta\subseteq V_\zeta$ and $\xi d\eta=0$ otherwise.
\begin{remark}
Notice that every $G$-anneid $A$ may be observed as an
$AG$-moduloid.
\end{remark}
\section{Homomorphisms and factor structures}
\noindent Throughout this section, $A$ denotes a $G$-anneid and
$M$ an $AG$-moduloid.
\begin{definition}
A nonempty subset $N$ of an $AG$-moduloid $M$ is called a
\emph{submoduloid} of $M$ if:
\begin{itemize}
    \item[$i)$] $(\forall x,y\in N)\ x\#y\Rightarrow x-y\in N;$
    \item[$ii)$] $(\forall x\in N)(\forall\alpha\in G)(\forall a\in A)\ x\alpha a\in N.$
\end{itemize}
A submoduloid of an $AG$-moduloid $A$ is called a \emph{right
ideal} of a $G$-anneid $A.$ We similarly define a \emph{left
ideal} and an \emph{ideal} (\emph{two-sided}) of $A.$
\end{definition}
\begin{definition}
An ordered triple $(\kappa,\varphi,\theta)$ is called a
\emph{quasihomomorphism of $AG$-moduloids} $M$ and $M'$ if:
\begin{itemize}
    \item[$i)$] $\kappa:M\to M'$ is a quasihomomorphism;
    \item[$ii)$] $\varphi:G\to G$ is a quasiisomorphism;
    \item[$iii)$] $\theta:A\to A$ is a quasiisomorphism;
    \item[$iv)$] $\kappa(m\alpha a)=\kappa(m)\varphi(\alpha)\theta(a)$
    $(m\in M, \alpha\in G, a\in A).$
\end{itemize}
It is called a \emph{homomorphism} if $\kappa,$ $\varphi,$
$\theta$ are moreover homomorphisms.
\end{definition}
\begin{definition}
An ordered pair $(\theta,\varphi)$ is called a
\emph{quasihomomorphism of $G$-anneids} $A$ and $A'$ if:
\begin{itemize}
    \item[$i)$] $\theta:A\to A'$ is a quasihomomorphism;
    \item[$ii)$] $\varphi:G\to G$ is a quasiisomorphism;
    \item[$iii)$] $\theta(a\alpha b)=\theta(a)\varphi(\alpha)\theta(b)$
    $(a,b\in A, \alpha\in G).$
\end{itemize}
It is called a \emph{homomorphism} if $\theta$ and $\varphi$ are
moreover homomorphisms.
\end{definition}
\noindent For future purposes, let us mention that, generally,
when we write $\sum_{j=1}^na_j,$ where $a_j$ are elements of a
$G$-anneid $A,$ it is assumed that $a_j$ are mutually addable.\\
If $I$ and $J$ are right (left, two-sided) ideals of a $G$-anneid
$A,$ then we define $I+J$ to be the set of all $x\in A$ such that
$x=a+b,$ where $a\in I,$ $b\in J.$ It is easy to verify that $I+J$
is also a right (left, two-sided) ideal of $A.$ It is also an easy
exercise to prove that the intersection of right (left, two-sided)
ideals of $A$ is again a right (left, two-sided) ideal of $A.$\\
The smallest right ideal which contains an element $a$ of a
$G$-anneid $A$ is called the \emph{principal right ideal}
generated by $a$ and is denoted by $|a\rangle.$ The
\emph{principal left} $\langle a|$ and the \emph{ideal} $\langle
a\rangle$ generated by $a$ are defined similarly.\\
Like in the case of gamma rings, it is clear that
$|a\rangle=\mathbb{Z}a+aGA,$ $\langle a|=\mathbb{Z}a+AGa,$ and
$\langle a\rangle=\mathbb{Z}a+aGA+AGa+AGaGA.$\\
If $I$ is an ideal of a $G$-anneid $A,$ then $A/I$ is a $G$-anneid
if we put \[(a+I)\alpha(b+I)=a\alpha b+I,\qquad (a+I, b+I\in
A/I)\]
and we call this $G$-anneid a \emph{factor $G$-anneid}.\\
If $N$ is an $AG$-submoduloid of an $AG$-moduloid $M,$ then we
define a \emph{factor $AG$-moduloid} $M/N$ similarly.
\begin{definition}
An $AG$-moduloid $M$ is called \emph{regular} if for all $x\in M,$
$0\neq x\alpha a\#x\beta b\neq0$ implies $\alpha\#\beta$ and
$a\#b.$ A gamma anneid $A$ is called \emph{right regular}
(\emph{left regular}) if it is regular as a right $AG$-moduloid
(left $AG$-moduloid). A gamma anneid is called \emph{regular} if
it is both left and right regular.
\end{definition}
\section{The Jacobson radical of a gamma anneid}
\noindent Jacobson radicals of $\Gamma$-rings are extensively
treated (see e.g. \cite{cl} and \cite{rs}), and here, we will
introduce and observe these radicals for gamma anneids by
following the concepts and results obtained for rings in
\cite{jac} and for anneids in \cite{ha} and in the final section
we will obtain the relations among these radicals. All results can
be easily translated to the language of graded gamma rings.\\
Let $M$ be an $AG$-moduloid, $N$ an $AG$-submoduloid and let $S$
be a subset of $M.$ If we define the set $(N:S)$ to be the set of
$a\in A$ such that $SGa\subseteq N,$ then $(N:S)$ is a right ideal
of $A.$ Indeed, let $b,b'\in (N:S),$ $b\#b',$ $\alpha,\beta\in G,$
$a\in A$ and $s\in S.$ Then $SG(b-b')=SGb-SGb'\subseteq N$ and
$(s\beta b)\alpha a\in N\alpha a\subseteq N.$
\begin{lemma}
If $M$ is a regular $AG$-moduloid and $x\in M,$ $\alpha\in G,$
then $A/(0:x)_\alpha\cong x\alpha A,$ where $(0:x)_\alpha=\{a\in
A\ |\ x\alpha a=0\}.$
\end{lemma}
\begin{proof}
Let $\kappa:A\to x\alpha A$ be the mapping defined with
$\kappa(a)=x\alpha a,$ $\varphi:G\to G$ and $\theta:A\to A$
identities. Obviously, $(\kappa,\varphi,\theta)$ is a homomorphism
and $\ker\kappa=(0:x)_\alpha.$
\end{proof}
\begin{definition}
An $AG$-moduloid $M$ is called \emph{irreducible} if $MGA$ is
nonzero and if $M$ has no nontrivial submoduloids.
\end{definition}
\noindent The following notion is inspired by the notions of a
\emph{primitive} and a \emph{largely primitive anneid} from
\cite{ha}.
\begin{definition}
A $G$-anneid $A$ is called \emph{primitive} if there exists an
irreducible regular and faithful $AG$-moduloid $M$
($(0:M)=\{0\}$). It is called \emph{largely primitive} if there
exists an irreducible and faithful $AG$-moduloid (whether regular
or not).
\end{definition}
\begin{definition}
An ideal $I$ of a $G$-anneid $A$ is called \emph{primitive}
(\emph{largely primitive}) if a $G$-anneid $A/I$ is primitive
(largely primitive).
\end{definition}
\noindent In \cite{ha}, the \emph{Jacobson} and the \emph{large
Jacobson radical of an anneid} are introduced, and we do the same
for $G$-anneids.
\begin{definition}
The \emph{Jacobson radical} (\emph{large Jacobson radical}) $J(A)$
of a $G$-anneid $A$ is the intersection of annihilators of all
irreducible regular $AG$-moduloids (irreducible $AG$-moduloids). A
$G$-anneid $A$ is called \emph{semisimple} if $J(A)=\{0\}.$
\end{definition}
\noindent As in the case of classical rings, one may prove the
following theorem, see \cite{jac}.
\begin{theorem}
The Jacobson radical of a gamma anneid is the intersection of its
primitive ideals.
\end{theorem}
\noindent In \cite{ha}, notions of strictly cyclic moduloids and
modular right ideals of anneids are introduced in order to
describe the Jacobson radical of an anneid. Here we do similarly
for $AG$-moduloids and $G$-anneids.
\begin{definition}
An $AG$-moduloid $M$ is called $\alpha$-\emph{strictly cyclic} if
there exists an element $x\in M$ such that $x\alpha A=M,$ where
$\alpha\in G.$ Such an element $x$ is called an
$\alpha$-\emph{strict generator} of $M.$ If $x\alpha A=M$ for
every $\alpha\in G,$ then $M$ is called \emph{strictly cyclic} and
$x$ is called a \emph{strict generator} of $M.$
\end{definition}
\begin{definition}
A right ideal $I$ of a $G$-anneid $A$ is called \emph{modular} if
there exist elements $u\in A$ and $\alpha\in G$ such that for all
$a\in A,$ $a$ and $u\alpha a$ are congruent modulo $I.$ Such an
element $u$ is called an $\alpha$-\emph{left identity modulo $I.$}
\end{definition}
\begin{definition}
Let $A$ be a $G$-anneid, $\Delta$ and $D$ grading sets
corresponding to $A$ and $G,$ respectively. An element
$\delta\in\Delta$ is called an $\alpha$-\emph{idempotent}, for
$0\neq\alpha\in G,$ if \[\delta d(\alpha)\delta=\delta,\] where
$d(\alpha)\in D$ is the degree of $\alpha.$ It is called
\emph{idempotent} if it is $\alpha$-idempotent for all $\alpha\in
G.$
\end{definition}
\begin{remark}\label{rem1}
To say that $u$ is an $\alpha$-left identity modulo $I,$ is the
same as to say that for all $a\in A,$ either $a$ and $u\alpha a$
belong to $I,$ or $a\#u\alpha a,$ $a-u\alpha a\in I.$ If moreover
$u\in I,$ we have that $I=A.$ Note also that if $I$ is a proper
ideal, then $\delta(u)$ is an $\alpha$-idempotent, that is,
$\delta(u)=\delta(u)d(\alpha)\delta(u).$ Indeed, $u\notin I$
implies $u\#u\alpha u$ and $u-u\alpha u\in I.$ Hence, $u\alpha
u\notin I$ and in particular, $u\alpha u\neq0.$
\end{remark}
\begin{theorem}
If $M$ is a strictly cyclic regular $AG$-moduloid, then $M\cong
A/I,$ where $I$ is a right modular ideal of a $G$-anneid $A.$ If
$I$ is a right modular ideal of a $G$-anneid $A,$ then
$I=(0:x)_\alpha,$ where $x$ is an $\alpha$-strict generator of an
$AG$-moduloid $M$, which is not necessarily regular.
\end{theorem}
\begin{proof}
We follow the proof presented for moduloids in \cite{ha}. Let $M$
be a strictly cyclic regular $AG$-moduloid with a strict generator
$x.$ Since $M=x\alpha A\cong A/(0:x)_\alpha,$ we only need to
prove that $(0:x)_\alpha$ is modular. Since $x\in x\alpha A$ and
since $M$ is regular, there exists $a\in A$ such that $x=x\alpha
a.$ Let $b\in A.$ Then $x\alpha b=(x\alpha a)\alpha b.$ If
$x\alpha b=x\alpha(a\alpha b)=0,$ then both $b$ and $a\alpha b$
belong to $(0:x)_\alpha.$ If $x\alpha b=x\alpha(a\alpha b)\neq0,$
then, since $M$ is regular, $b\#a\alpha b$ and hence
$x\alpha(b-a\alpha b)=0,$ so $b-a\alpha b\in(0:x)_\alpha.$ Thus,
$(0:x)_\alpha$ is modular.\\
Conversely, if $I$ is a right modular ideal of a $G$-anneid $A$
with $a$ as an $\alpha$-left identity modulo $I,$ then $a+I$ is an
$\alpha$-strict generator of $A/I$ and $I=(0:a+I)_\alpha.$
\end{proof}
\begin{cor}
If $I$ is a right modular ideal of $A,$ then $I$ contains $(I:A).$
\end{cor}
\noindent Following propositions can be proved by similar
arguments given for the case of classical rings \cite{jac}.
\begin{proposition}
Every right modular ideal is contained in a maximal right ideal.
\end{proposition}
\begin{proposition}
An $AG$-moduloid $M$ is irreducible if and only if:
\begin{itemize}
    \item[$a)$] $M\neq\{0\};$
    \item[$b)$] Every nonzero element of $M$ is a strict generator
    of $M.$
\end{itemize}
\end{proposition}
\begin{proposition}
A regular $AG$-moduloid $M$ is irreducible if and only if $M\cong
A/I$ for some right modular maximal ideal $I$ of $A.$
\end{proposition}
\begin{cor}
Every primitive ideal $J$ of a $G$-anneid $A$ has the form $(I:A)$
for some right modular maximal ideal $I$ of $A.$ Conversely, if
$I$ is a right modular maximal ideal of a $G$-anneid $A,$ then
$(I:A)$ is a large primitive ideal (primitive if $A$ is right
regular).
\end{cor}
\begin{proposition}
The Jacobson radical of a right regular $G$-anneid is the
intersection of its right modular maximal ideals.
\end{proposition}
\section{Structure of the Jacobson radical of a regular gamma anneid}
\noindent Let $A$ be a regular $G$-anneid and $J(A)$ its Jacobson
radical. In \cite{ha}, it is proved that all left identities
modulo a proper right modular ideal of an anneid have the same
degree, and here we have the analogous result.
\begin{lemma}
If $I$ is a proper right modular ideal of $A,$ all left identities
modulo $I$ have the same degree.
\end{lemma}
\begin{proof}
Let $a$ be an $\alpha$-identity and $b$ a $\beta$-identity modulo
$I.$ Hence, for all $x\in A,$ we have $x+I=a\alpha x+I=b\beta
x+I.$ If $x\notin I,$ $a\alpha x+I=b\beta x+I\neq0+I,$ which
implies $a\alpha x\#b\beta x,$ and since $A$ is regular, $a\#b$
and $\alpha\#\beta.$
\end{proof}
\begin{definition}
The degree of all left identities modulo $I$ is called the
\emph{degree} of $I.$
\end{definition}
\noindent Let $e\in\Delta^*$ be an element whose degree
$\delta(e)$ is an $\alpha$-idempotent of $\Delta^*,$ that is,
$\delta(e)d(\alpha)\delta(e)=\delta(e).$ Notice that $A(e)$ is
then a $G(\alpha)$-ring. Indeed, let $x,y\in A(e).$ Then $x\#e,$
$y\#e$ and $e\#e\alpha e,$ hence, $x\alpha y\#e$ (see Lemma
\ref{lemma1}). For the sake of simplicity, we will denote
$\delta(e)$ also by $e$ in the sequel.\\
It is stated in \cite{ha1} and proved in \cite{ha} that there
exists a one-to-one correspondence between the maximal modular
right ideals of a regular anneid $A$ of degree $e$ and the maximal
modular right ideals of $A(e).$ Here, we have the following
result.
\begin{theorem}\label{theorem1}
Let $A$ be a regular $G$-anneid and $e$ an $\alpha$-idempotent of
$\Delta^*.$ To every right modular maximal ideal $I$ of $A$ with
degree $e,$ let us assign $I_e=I\cap A(e).$ Also, to every right
maximal ideal $S$ of $A(e),$ let us assign the set $\hat{S}$ of
elements $x\in A$ such that $xGA\cap A(e)\subseteq S.$ This
establishes a one-to-one correspondence between the set of right
modular maximal ideals of $A$ with degree $e$ and the set of right
modular maximal ideals of $A(e).$
\end{theorem}
\begin{proof}
We follow the proof of the corresponding theorem for anneids from
\cite{ha}. Let $S$ be a right modular maximal ideal of $A(e)$ and
let $u$ be an $\alpha$-left identity modulo $S.$ It is clear that
$\hat{S}$ is a right ideal of $A.$ Let $s\in S,$ $a\in A,$
$\beta\in G,$ and assume that $0\neq s\beta a\in A(e).$ Then
\[ed(\alpha)e=e=\delta(s\beta a)=\delta(s)d(\beta)\delta(a).\]
Since $\delta(s)=e,$ and since $A$ is regular, it follows that
$\delta(a)=e.$ Hence, $s\beta a\in S,$ which means that
$S\subseteq \hat{S}\cap A(e).$ Conversely, let $x\in\hat{S}\cap
A(e)$ be such that $x\notin S.$ Since $S$ is maximal, we have
$A(e)=S+|x\rangle,$ and so $u=s+nx+x\beta a,$ where $s\in S,$
$n\in\mathbb{Z},$ $\beta\in G,$ $a\in A.$ Since $x\alpha a,$
$x\alpha x\in xGA\cap A(e)\subseteq S$ and $u\alpha x=s\alpha
x+nx\alpha x+x\beta a\alpha x,$ we have $u\alpha x\in S,$ and
since $u$ is an $\alpha$-identity modulo $S,$ $x\in S.$ Thus,
$\hat{S}\cap A(e)\subseteq S.$ We have proved that $\hat{S}\cap A(e)=S.$\\
Next we want to prove that $u$ is an $\alpha$-left identity modulo
$\hat{S}.$ Let $a\in A$ be such that $a\in\hat{S}.$ It suffices to
show that $u\alpha a\in\hat{S}.$ Let $b\in A$ and $\beta\in G$
such that $0\neq(u\alpha a)\beta b\in A(e).$ Since $A$ is regular,
$a\beta b\in A(e),$ and since $a\in\hat{S},$ we have $a\beta b\in
S$ which implies $u\alpha(a\beta b)\in S.$ Assume now that
$a\notin\hat{S}.$ Then there exist $b\in A$ and $\beta\in G$ such
that $a\beta b$ belongs to $A(e)$ but not to $S.$ Since $a\beta
b-u\alpha(a\beta b)\in S,$ $u\alpha(a\beta b)\notin S$ and in
particular, $u\alpha(a\beta b)\neq 0.$ Since $a\beta
b\#u\alpha(a\beta b),$ by regularity of $A$ we have $a\#u\alpha
a.$ If $x\in A$ and $\gamma\in G$ are elements such that
$0\neq(a-u\alpha a)\gamma x\in A(e),$ $a\gamma x\in A(e),$ hence
we have $(a-u\alpha a)\gamma x\in S$ which implies $a-u\alpha
a\in\hat{S}.$\\
Now, let $B$ be a proper right ideal which contains $\hat{S}.$
Then $B\cap A(e)\supseteq\hat{S}\cap A(e)=S.$ Since $B$ is proper
and $u$ an $\alpha$-left identity modulo $\hat{S},$ $A(e)$ is not
a subset of $B.$ Hence, $B\cap A(e)=S.$ Now it is easy to verify
that $B=\hat{S}.$\\
Let $I$ be a right modular maximal ideal of $A$ of degree $e$ and
let $u$ be an $\alpha$-left identity modulo $I.$ It is easy to
verify that $I_e=I\cap A(e)$ is a right ideal of $A(e),$ $u$ an
$\alpha$-left identity modulo $I_e,$ and that $I=\hat{I_e}.$
\end{proof}
\noindent As in the case of anneids from \cite{ha1,ha}, we have
the following result.
\begin{theorem}\label{j1}
The Jacobson radical of a regular $G$-anneid $A$ consists of
elements $x\in A$ such that $xGA\cap A(e)\subseteq J(A(e)),$ for
every $\alpha$-idempotent $e\in\Delta.$
\end{theorem}
\begin{proof}
According to the previous theorem, $x\in A$ is in the intersection
of all right modular maximal ideals of $A$ with degree $e$ if and
only if $x\in\hat{S}$ for every right modular maximal ideal $S$ of
$A(e),$ that is, if $xGA\cap A(e)\subseteq S\subseteq J(A(e)).$
\end{proof}
\noindent In \cite{ha}, the notion of a \emph{quasi-regular
element of an anneid} is introduced, and here we do similar for
gamma anneids.
\begin{definition}
An element $z$ of a $G$-anneid $A$ is called $\alpha$-\emph{right
quasi-regular} or shortly $\alpha$-\emph{rqr} if there exists no
proper right ideal of $A$ such that $z$ is an $\alpha$-left
identity modulo that ideal. If $z$ is $\alpha$-right quasi-regular
for all $\alpha\in G,$ then $z$ is called \emph{right
quasi-regular}. It is clear how to define a \emph{left
quasi-regular element}. An element is called \emph{quasi-regular}
if it is both left and right quasi-regular. A right ideal of $A$
is called \emph{quasi-regular} if all its elements are right
quasi-regular. We similarly define left and two-sided
quasi-regular ideals.
\end{definition}
\begin{proposition}\label{regular}
Let $A$ be a regular $G$-anneid. An element $z\in A$ is
$\alpha$-rqr (rqr) if and only if one of the following is
satisfied:
\begin{itemize}
    \item[$i)$] $\delta(z)$ is not an $\alpha$-idempotent of
    $\Delta^*$ (for all $\alpha\in G$);
    \item[$ii)$] $e=\delta(z)$ is an $\alpha$-idempotent of $\Delta^*$ (for all $\alpha\in G$) and
    $z$ is an $\alpha$-right quasi-regular element (right quasi-regular element) of an $G(\alpha)$-ring
    $A(e)$ ($\Gamma$-ring $A(e)$).
\end{itemize}
\end{proposition}
\begin{proof}
If $i)$ holds, then $z$ is $\alpha$-rqr according to Remark
\ref{rem1}. If $e=\delta(z)$ is an $\alpha$-idempotent of
$\Delta^*,$ it is enough to prove that then $z$ is $\alpha$-rqr in
$A$ if and only if it is $\alpha$-rqr in $A(e),$ but this follows
from Theorem \ref{theorem1}.
\end{proof}
\noindent By classical means one may prove the following
corollary, see \cite{jac}.
\begin{cor}
If $I$ is a right quasi-regular ideal of a $G$-anneid $A,$ then
every element $z\in I$ is quasi-regular.
\end{cor}
\begin{proposition}
The Jacobson radical of a regular $G$-anneid $A$ is a
quasi-regular ideal which contains all right quasi-regular ideals
of $A.$
\end{proposition}
\begin{proof}
We follow the proof of the corresponding proposition for anneids
given in \cite{ha}. Let $z\in J(A)$ and let us assume that $z$ is
an $\alpha$-left identity modulo $I,$ where $I$ is a proper right
ideal of $A.$ Without loss of generality we may assume that $I$ is
maximal. Since $J(A)$ is the intersection of all of its right
modular maximal ideals, it follows that $z\in I,$ which
would mean that $I$ is not a proper ideal.\\
Suppose now that $I$ is a right quasi-regular ideal of $A$ and
$z\in I.$ Also, let $M$ be a regular irreducible $AG$-moduloid and
$z\notin(0:M).$ Then there exist elements $x\in M$ and $\gamma\in
G$ such that $x\gamma z\neq0.$ Hence, $x\gamma z$ is a strict
generator of $M$ and so there exist $a\in A$ and $\beta\in G$ such
that $x=(x\gamma z)\beta a.$ Thus, $x\gamma (z\beta a)=x\gamma
z\beta a\gamma z\beta a.$ By regularity we have that
$e=\delta(z\beta a)$ is a $\gamma$-idempotent of $\Delta^*.$
Hence, $z\beta a$ is a $\gamma$-right quasi-regular in $A(e).$ By
\cite{rs}, we know that there exists $t\in A(e)$ such that $z\beta
a+t-(z\beta a)\gamma t=0,$ which implies $0=x-x\gamma (z\beta
a)-(x-x\gamma (z\beta a))\gamma t=x-x\gamma(z\beta a+t-(z\beta
a)\gamma t)=x.$ Hence, $z\in J(A).$
\end{proof}
\begin{cor}
The Jacobson radical and the left Jacobson radical of a regular
$G$-anneid are equal.
\end{cor}
\begin{cor}\label{ji}
The Jacobson radical of a $G$-anneid $A$ consists of all elements
$x\in A$ such that $x\alpha a$ is right quasi-regular element for
all $\alpha\in G,$ $a\in A.$
\end{cor}
\begin{proposition}\label{j2}
If $A$ is a regular $G$-anneid and $e$ a $\beta$-idempotent of
$\Delta^*,$ then $J(A(e))=J(A)\cap A(e).$
\end{proposition}
\begin{proof}
We may follow the proof of the corresponding result for regular
anneids from \cite{ha}. It is clear that $J(A)\cap A(e)$ is a
quasi-regular ideal of $A(e),$ hence $J(A)\cap A(e)\subseteq
J(A(e)).$ Conversely, if $z\in J(A(e)),$ assume that $z\notin
J(A).$ Then there exist $x\in A,$ $\delta=\delta(x),$ and
$\alpha\in G$ such that $\delta(z\alpha x)$ is an idempotent $f$
of $\Delta^*$ and $z\alpha x\notin J(A(f)).$ Using the regularity,
it is easy to verify that $e,$ $\delta$ and $f$ are mutually
distinct. Like in the case of rings, if we observe a $\Gamma$-ring
$B$ and a set $I$ of elements $b\in B$ such that $B\Gamma b\Gamma
B=\{0\},$ then $I$ is an ideal of $B.$ Moreover, $I\Gamma I\Gamma
I=\{0\},$ and hence $I\subseteq J(B).$ In our case, $z\alpha
x\notin J(A(f)),$ and hence there exist $t\in A(f)$ and $\gamma\in
G$ such that $t\gamma z\alpha x\neq 0,$ which implies
$(fd(\gamma)e)d(\alpha)\delta=fd(\gamma)(ed(\alpha)\delta)=fd(\gamma)f=f=ed(\alpha)\delta\neq0,$
and then regularity implies $fd(\gamma)e=e=ed(\beta)e,$ hence
$f=e,$ again by regularity.
\end{proof}
\noindent The following proposition is an analogue of the similar
result for regular anneids from \cite{ha}.
\begin{proposition}
Let $A$ be a regular $G$-anneid and $I$ a right ideal of $A.$ Then
$J(I)=\{x\in I\ |\ xGI\subseteq J(A)\}.$
\end{proposition}
\begin{proof}
Let $K=\{x\in I\ |\ xGI\subseteq J(A)\}$ and $x\in K,$ $\alpha\in
G,$ $a\in I.$ We claim that $x\alpha a$ is right quasi-regular in
$I.$ According to Proposition \ref{regular}, we may assume that
$(x\alpha a)\gamma(x\alpha a)\neq0$ and addable with $x\alpha a,$
for all $\gamma\in G,$ since every nilpotent element is right
quasi-regular (nilpotent element is defined as in the case of
ordinary gamma rings). Hence, $e=\delta(x\alpha a)$ is an
idempotent of $\Delta^*$ and $x\alpha a\in J(A)$ is right
quasi-regular in $A(e).$ This means that there exists $z\in A(e)$
such that $x\alpha a+z-(x\alpha a)\gamma z=0$ for all $\gamma\in
G,$ and hence, $z\in I.$ Thus, $x\alpha a$ is rqr in $I,$ and
according to Corollary \ref{ji}, $x\in J(I).$ Hence, $K\subseteq
J(I).$ Conversely, it suffices to show that $J(I)GI\subseteq
J(A),$ and since $J(I)GI$ is a right ideal of $A,$ it is enough to
prove that every $x\in J(I)GI$ is rqr in $A.$ Again, we may assume
that $x\gamma x\neq0$ and addable with $x,$ for all $\gamma\in G.$
Then $e=\delta(x)$ is an idempotent of $\Delta^*.$ Element $x$ is
rqr in $I(e),$ and hence for all $\gamma\in G,$ there exists $y\in
I(e)$ such that $x+y-x\gamma y=0,$ and hence, $x$ is rqr in $A.$
\end{proof}
\begin{cor}
For an ideal $I$ of a regular $G$-anneid $A,$ $J(I)=I\cap J(A).$
\end{cor}
\section{Relation between the large Jacobson radical and the Jacobson radical}
\noindent Let $A$ be a $G$-anneid and $R$ the corresponding graded
$\Gamma$-ring. In the next proposition we give some connections
between the Jacobson radical of $A,$ the large Jacobson radical of
$A,$ and the Jacobson radical of $R.$ For the corresponding
relations for anneids, see \cite{ha}. By the Jacobson radical of
$R$ we mean the notion as defined in \cite{rs}.
\begin{proposition}
If $J(A)$ is the Jacobson radical of $A,$ $J_l(A)$ the large
Jacobson radical of $A$ and $J(R)$ the Jacobson radical of $R,$
then:
\begin{itemize}
    \item[$i)$] $J_l(A)=J(R)\cap A;$
    \item[$ii)$] $J_l(A)\subseteq J(A).$
\end{itemize}
\end{proposition}
\begin{proof}
$i)$ Every simple $R\Gamma$-module $V$ can be viewed as an
irreducible $AG$-moduloid, and hence $J_l(A)\subseteq J(R).$ Thus,
$J_l(A)\subseteq J(R)\cap A.$ Conversely, suppose $a\in J(R)\cap
A.$ It is enough to prove that $a\in(0:M)$ if $M$ is an
irreducible $AG$-moduloid. Assume that $a\notin(0:M).$ Then there
exist $x\in M,$ $\alpha\in G$ and $a\in A$ such that $x\alpha
a\neq0.$ Since $M$ is an irreducible $AG$-moduloid, $x\alpha a$ is
a strict generator and it follows that $\overline{M}=x\alpha
a\gamma R,$ for all $\gamma\in\Gamma.$ Let $m\in R$ be such that
$x=x\alpha a\alpha m.$ Then, for all $y\in R$ and $\beta\in\Gamma$
we have $x\beta y=x\alpha a\alpha m\beta y,$ and particularly,
$x\alpha(y-a\alpha m\alpha y)=0,$ that is $y-a\alpha m\alpha
y\in(0:x)_\alpha.$ Since $a\in J(R),$ $a\alpha m$ also belongs to
$J(R).$ For all $\beta\in\Gamma$ there exists $z\in R$ such that
$a\alpha m+z-a\alpha m\beta z=0$ which implies $z-a\alpha m\beta
z=-a\alpha m,$ and since $y-a\alpha m\alpha y\in(0:x)_\alpha,$ for
all $y\in R,$ we have $a\alpha m\in(0:x)_\alpha.$ Thus, $x=x\alpha
a\alpha m=0,$ a contradiction.\\
$ii)$ It is obvious that $J_l(A)\subseteq J(A).$
\end{proof}
\section*{Acknowledgement}
\noindent The author would like to express his big gratitude to
Professor Mirjana Vukovi\'{c} who introduced him to the theory of
general graded rings and for providing him with the substantial
amount of the literature, in particular, with \cite{ha}.

\noindent Emil Ili\'{c}-Georgijevi\'{c}\\
Faculty of Civil Engineering, University of Sarajevo\\
Patriotske lige 30, 71000 Sarajevo, Bosnia and Herzegovina\\
E-mail: emil.ilic.georgijevic@gmail.com

\begin{thebibliography}{99}
\bibitem{wb} W.~E.~Barnes, On the $\Gamma $-rings of Nobusawa, Pacific J.
Math. \textbf{18} (1966), 411--422.
\bibitem{bg} G.~L.~Booth, N.~J.~Groenewald, $\Gamma $-rings and normal
radicals, Period. Math. Hung. \textbf{31} (1) (1995), 5--10.
\bibitem{cha} M.~Chadeyras, Essai d'une th\'eorie noetherienne pour les anneaux commutatifs, dont
la graduation est aussi g\'en\'erale que possible, {\it Bull. de
la S.M.F.}, Suppl\'ement, M\'emoire No. {\bf 22}, Paris 1970.
\bibitem{cl} W.~E.~Coppage, J.~Luh, Radicals of gamma rings, J. Math. Soc.
Japan \textbf{21} (1971), 40--52.
\bibitem{flin} F.~Fusheng, G.~Lingzhong, Group graded gamma rings,
Northeast. Math. J. \textbf{14} (2) (1998), 177--186.
\bibitem{ha1} E.~Halberstadt,\ Le radical d'un anneide r\'egulier, C. R. Acad. Sci., Paris, S\'er. A, Paris \textbf{270} (1970),
361--363.
\bibitem{ha} E.~Halberstadt, Th\'eorie artinienne homog\`ene des anneaux gradu\'es \`a grades non commutatifs r\'eguliers, Th\`ese doct. sci.
math., Arch. orig. Cent. Doc. C.N.R.S., no \textbf{5962}, Paris:
Centre National de la Recherche Scientifique, 182 p, 1971.
\bibitem{jac} N.~Jacobson, Structure of rings, American Mathematical Society Colloquium Publications, Vol.
\textbf{37}, 1964.
\bibitem{ak1} A.~V.~Kelarev, On groupoid graded rings, J. Algebra
178 (1995), 391--399.
\bibitem{ak} A.~V.~Kelarev, \emph{Ring constructions and
applications}, Series in Algebra, Vol. \textbf{9}, World
Scientific, 2002.
\bibitem{ak2} A.~V.~Kelarev, A.~Plant, Bergman's lemma for graded
rings, Comm. Algebra 23(12) (1995), 4613--4624.
\bibitem{ug} M.~Krasner, Une g\'en\'eralisation de la notion de corps--corpo\'ide. Un corpo\'ide remarquable de la th\'eorie des corps
valu\'es, C. R. Acad. Sci. Paris \textbf{219} (1944), 345--347.
\bibitem{agg} M.~Krasner, Anneaux gradu\'es g\'en\'eraux, Colloque d'Alg\`ebre Rennes (1980), 209--308.
\bibitem{sp} M.~Krasner,~M.~Vukovi\'{c}, Structures paragradu\'ees (groupes, anneaux,
modules), Queen's Papers in Pure and Applied Mathematics, No.
\textbf{77}, Queen's University, Kingston, Ontario, Canada 1987.
\bibitem{ks} S.~Kyuno, Prime ideals in gamma rings, Pacific J.
Math. \textbf{98} (2) (1982), 375--379.
\bibitem{nob} N.~Nobusawa, On a generalization of the ring theory,
Osaka J. Math. \textbf{1} (1964), 81--89.
\bibitem{naoy} C.~N{\u{a}}st{\u{a}}sescu, F.~Van Oystaeyen, Methods of graded
rings, Lecture Notes in Mathematics, \textbf{1836},
Springer-Verlag, Berlin, 2004.
\bibitem{rs} T.~S.~Ravisankar, U.~S.~Shukla, Structure of
$\Gamma$-rings, Pacific J. Math. \textbf{80} (2) (1979), 537--559.
\bibitem{mv} M.~Vukovi\'{c}, Structures gradu\'ees et paragradu\'ees,
{\it Prepublication de l'Institut Fourier}, Universit\'e de
Grenoble I, {\bf 536} (2001), pp. 1--40.
\bibitem{dw} D.~Wang, On the Brown-McCoy property for
$\Gamma$-rings, Comm. Algebra \textbf{24} (2) (1996), 477--486.
\end{thebibliography}
\end{document}